\documentclass[11pt]{amsart}

\input xy
\xyoption{all}

\newcommand{\IR}{\mathbb R}
\newcommand{\U}{\mathcal U}
\newcommand{\V}{\mathcal V}

\newcommand{\IQ}{\mathbb Q}

\newcommand{\IN}{\mathbb N}
\newcommand{\cl}{\mathrm{cl}}
\newcommand{\intr}{\mathrm{int}}

\newcommand{\FU}{{\mathcal Q}}
\newcommand{\fe[1]}{\FU_{#1}}
\newcommand{\rcluwL}{{}^\circ\kern-2pt\overline{uw}L}

\newtheorem{theorem}{Theorem}
\newtheorem{corollary}{Corollary}
\newtheorem{proposition}{Proposition}
\newtheorem{problem}{Problem}

\theoremstyle{definition}
\newtheorem{definition}{Definition}

\newtheorem{remark}{Remark}

\title{Each regular paratopological group is completely regular}
\author{Taras Banakh and Alex Ravsky}
\address{T.Banakh: Ivan Franko National University of Lviv (Ukraine), and Jan Kochanowski University in Kielce (Poland)}
\email{t.o.banakh@gmail.com}
\address{A.Ravsky: Pidstryhach Institute for Applied Problems of Mechanics and Mathematics of National Academy of Sciences, Lviv, Ukraine}
\email{oravsky@mail.ru}

\begin{document}
\begin{abstract} We prove that a semiregular topological space $X$ is completely regular if and only if its topology is generated by a normal quasi-uniformity. This characterization implies that each regular paratopological group is completely regular.
This resolves an old problem in the theory of paratopological groups, which stood open for about 60 years. Also we define a natural uniformity on each paratopological group and using this uniformity prove that each (first countable) Hausdorff paratopological group is functionally Hausdorff (and submetrizable). This resolves another two known open problems in the theory of paratopological groups.
\end{abstract}
\maketitle

This paper was motivated by an old unsolved problem in the theory of paratopological groups, asking if each regular paratopological group is completely regular (see \cite[Question 1.2]{Rav1}, \cite[Problem 1.3.1]{AT}, and \cite[Problem 2.1]{Tka}). Discussing this problem in the survey \cite[\S2]{Tka}, Tkachenko writes that it is ``open for about 60 years'' and ``the authority of the question is unknown, even if all specialists in the area have it in mind''. Six sections later in \cite[\S8]{Tka} Tkachenko writes that ``it is even more surprising that the following problem of Arhangel'skii \cite[3.11]{Ar} is open: {does every Hausdorff first countable paratopological group admit a weaker metrizable topology?}''

In this paper we will give surprisingly simple affirmative answers to both these open problems. In fact, our principal results hold not only for paratopological groups but also for topological monoids with open shifts and more generally for topological spaces whose topology is generated by a normal quasi-uniformity.

Now let us recall the necessary definitions related to quasi-uniformities (see \cite{FL}, \cite{Ku1} and  \cite{Ku2} for more information).

For a subset $A\subset X$ of a topological space by $\cl_X A$  and $\intr_X A$ we denote the closure and interior of $A$ in $X$. The sets $\cl_X A$, $\intr_X A$ and $\intr_X\cl_X A$ will be also denoted by $\overline{A}$, $A^\circ$, and $\overline{A}^\circ$, respectively.

Let $X$ be a topological space. A subset $U\subset X\times X$ containing the diagonal $\Delta_X=\{(x,y)\in X\times X:x=y\}$ is called an {\em entourage} on $X$.
Given two entourages $U,V\subset X\times X$ let
$$UV=\big\{(x,z)\in X\times X:\mbox{$\exists y\in X$ such that $(x,y)\in U$ and $(y,z)\in V$}\big\}$$be their composition and
$U^{-1}=\{(y,x):(x,y)\in U\}$ be the inverse entourage to $U$.

For a point $x\in X$ the set $B(x;U)=\{y\in X:(x,y)\in U\}$ is called the {\em $U$-ball} centered at $x$, and for a subset $A\subset X$ the set $B(A;U)=\bigcup_{a\in A}B(a;U)$ is the {\em $U$-neighborhood} of $A$.
It will be convenient to denote the sets $\overline{B(A;U)}$, $B(A;U)^\circ$ and $\overline{B(A,U)}^\circ$ by $\overline{B}(A;U)$, $B^\circ(A;U)$ and $\overline{B}^\circ\kern-2pt(A;U)$, respectively.

A {\em quasi-uniformity} on a set $X$ is a family $\U$ such that:
\begin{itemize}
\item for any entourages $U,V\in\U$ there is an entourage $W\in\U$ such that $W\subset U\cap V$;
\item for any $U\in\U$ there is $V\in\U$ such that $VV\subset U$;
\item for every entourage $U\in\U$, any entourage $V\subset X\times X$ containing $U$ belongs to $\U$.
\end{itemize}
A quasi-uniformity $\U$ on $X$ is called a {\em uniformity} if $\U=\U^{-1}$ where $\U^{-1}=\{U^{-1}:U\in\U\}$.

A subfamily $\mathcal B\subset\U$ is called a {\em base} of a quasi-uniformity $\U$ if each entourage $U\in\U$ contains some entourage $B\in\mathcal B$.
Each quasi-uniformity $\U$ on a set $X$ generates a topology $\tau_\U$ on $X$ consisting of sets $W\subset X$ such that for each $x\in W$ there is $U\in\U$ such that $B(x;U)\subset W$. Then for every point $x\in X$ the family $\{B(x;U)\}_{U\in\U}$ is a neighborhood base at $x$. It is clear that the topological space $(X,\tau_\U)$ satisfies the separation axiom $T_1$ if and only if the quasi-uniformity $\U$ is {\em separated} in the sense that $\bigcap\U=\Delta_X$.
It is known (see \cite{Ku1} or \cite{Ku2}) that the topology of any topological space $X$ is generated by a suitable quasi-uniformity. We shall say that a quasi-uniformity $\U$ on a topological space $X$ is {\em continuous} if the topology $\tau_\U$ generated by $\U$ is contained in the topology of $X$.

\begin{definition}
A quasi-uniformity $\U$ on a topological space $X$ is defined to be {\em normal\/} if $\overline{A}\subset \overline{B(A;U)}^\circ$ for any subset $A\subset X$ and any entourage $U\in\U$.

A topological space $X$ is called {\em normally quasi-uniformizable} if the topology of $X$ is generated by a normal quasi-uniformity.
\end{definition}

\begin{proposition}\label{p:uT} Each continuous uniformity $\U$ on a topological space $X$ is normal.
\end{proposition}

\begin{proof} Given a subset $A\subset X$ and an entourage $U\in\U$, we should prove that $\overline{A}\subset\overline{B(A;U)}^\circ$. Choose any entourage $V\in\U$ such that $V^{-1}V\subset U$. It follows that for every point $a\in\overline{A}$ the ball $B(a,V)$ meets the set $A$ and hence $a\in B(A;V^{-1})$. Then $B(a;V)\subset B(B(A;V^{-1});V)=B(A;V^{-1} V)\subset B(A;U)\subset \overline{B(A;U)}$ and hence $a\in\overline{B(A;U)}^\circ$.
\end{proof}

In the proof of the following theorem we use the classical method of the proof of Urysohn's Lemma.

\begin{theorem}\label{main} If $\U$ is a normal quasi-uniformity on a topological space $X$, then for any non-empty subset $A\subset X$ and entourage $U\in\U$ there exists a continuous function $f:X\to [0,1]$ such that
$A\subset f^{-1}(0)\subset f^{-1}\big([0,1)\big)\subset \overline{B(A;U)}^\circ$.
\end{theorem}

\begin{proof} Choose inductively a sequence of entourages $(U_n)_{n=0}^\infty\in\U^{\IN}$ such that $U_0\subset U$ and  $U_{n}U_{n}\subset U_{n-1}$ for every $n\in\IN$.

Let $\IQ_2=\{\frac{k}{2^n}:k,n\in\IN,\;0<k<2^n\}$ be the set of binary fractions in the interval $(0,1)$. Each number $r\in \IQ_2$ can be uniquely written as the sum $r=\sum_{n=1}^{\infty}\frac{r_n}{2^n}$ for some binary sequence $(r_n)_{n=1}^\infty\in\{0,1\}^{\IN}$ containing finitely many units. The sequence $(r_n)_{n\in\IN}$ will be called
the {\em binary expansion} of $r$. Since $r>0$, the number $l_r=\max\{n\in\IN:r_n=1\}$ is well-defined. So, $r=\sum_{i=1}^{l_r}\frac{r_i}{2^i}$.

For an entourage $V\in\U$ we put $V^1=V$ and $V^0=\Delta_X$.
For every $r\in\IQ_2$ consider the entourage $$U^r=U_1^{r_1}\cdots U_{l_r}^{r_{l_r}}\in\mathcal B,$$
which determines the closed neighborhood $\overline{B}(A;U^r)$ of $A$.

We claim that for any rational numbers $q<r$ in $\IQ_2$ the neighborhood $\overline{B}(A;U^q)$ is contained in the interior $\overline{B}^\circ\kern-2pt(A;U^r)$ of the neighborhood  $\overline{B}(A;U^r)$. Let $(q_n)$ and $(r_n)$ be the binary expansions of $q$ and $r$, respectively. The inequality $q<r$ implies that there exists a number $l\in\IN$ such that $0=q_l<r_l=1$ and $q_i=r_i$ for all $i<l$.
It follows that $l_q\ne l\le l_{r}$. If $l_q<l$, then the normality of $\U$ implies
$$
\begin{aligned}
\overline{B}(A;U^q)&=\overline{B}(A;U_1^{q_1}\cdots U_{l_q}^{q_{l_q}})=\overline{B}(A;U_1^{r_1}\cdots U_{l_q}^{r_{l_q}})\subset \overline{B}(A;U_1^{r_1}\cdots U_{l-1}^{r_{l-1}})\subset\\
 &\subset \overline{B}^\circ\kern-1pt(A;U_1^{r_1}\cdots U_{l-1}^{r_{l-1}}U^{r_l}_l)
 \subset\overline{B}^\circ(A;U_1^{r_1}\cdots U_{l_r}^{r_{l_r}})=\overline{B}^\circ\kern-2pt(A;U^r).
\end{aligned}
$$
If $l<l_q$, then the inclusions $U_{k}U_k\subset U_{k-1}$ for $l<k\le l_q+1$, guarantee that $U_{l+1}\cdots U_{l_q}U_{l_q+1}\subset U_l$ and then
\begin{multline*}
\overline{B}(A;U^q)=\overline{B}(x;U_1^{q_1}\cdots U_{l_q}^{q_{l_q}})\subset \overline{B}^\circ(A;U_1^{q_1}\cdots U_{l_q}^{q_{l_q}}U_{l_{q}+1})=\\
=\overline{B}^\circ\big(A;U_1^{q_1}\cdots U^{q_{l-1}}_{l-1}U^0_l\, U_{l+1}^{q_{l+1}}\cdots U_{l_q}^{q_{l_q}}U_{l_q+1}\big)\subset \overline{B}^\circ(A;U_1^{q_1}\cdots U_{l-1}^{q_{l-1}}\, U_l)=\\
=\overline{B}^\circ(A;U_1^{r_1}\cdots U_{l-1}^{r_{l-1}}U_l^{r_l})\subset \overline{B}^\circ(A;U_1^{r_1}\cdots U_{l_{r}}^{r_{l_{r}}})=\overline{B}^\circ(A;U^r).
\end{multline*}
So, $\overline{B}(A;U^q)\subset\overline{B}^\circ(A;U^r)$.

 Now define the function $f:X\to[0,1]$ by the formula
$$f(z)=\inf\big(\{1\}\cup \{q\in\IQ_2:z\in \overline{B}(A;U^q)\}\big)\mbox{ \ for \ } z\in X.
$$

It is clear that $A\subset f^{-1}(0)$ and $f^{-1}\big([0,1)\big)\subset \bigcup_{q\in\IQ_2}\overline{B}(A;U^q)=\bigcup_{r\in\IQ_2}\overline{B}^\circ\kern-2pt(A;U^r)\subset \overline{B}^\circ(A;U_0)\subset \overline{B}^\circ(A;U)$.
To prove that the map $f:X\to [0,1]$ is continuous, it suffices to check that for every real number $a\in (0,1)$ the sets $f^{-1}\big([0,a)\big)$ and $f^{-1}\big((a,1]\big)$ are open. This follows from the equalities $$f^{-1}\big([0,a)\big)=\bigcup_{\IQ_2\ni q<a}\overline{B}^\circ(A;U^q)\mbox{ \ and \ } f^{-1}\big((a,1]\big)=\bigcup_{\IQ_2\ni r>a}X\setminus \overline{B}(A;U^r).$$
\end{proof}

\begin{corollary}\label{cmain} If a topological space $X$ is normally quasi-uniformizable, then for any point $x\in X$ and  neighborhood $O_x\subset X$ of $x$  there exists a continuous function $f:X\to [0,1]$ such that $f(x)=0$ and
$f^{-1}\big([0,1)\big)\subset \overline{O}_x^\circ$.
\end{corollary}

We will apply Corollary~\ref{cmain} to establish equivalences of certain separation axioms in normally quasi-uniformizable spaces.

To avoid a possible ambiguity, let us fix the definitions of some separation axioms.
\begin{definition}
A topological space $X$ is called
\begin{itemize}
\item[$T_0$:] a {\em $T_0$-space} if  for any distinct points $x,y\in X$ there is an open set $U\subset X$ containing exactly one of these points;
\item[$T_1$:] a {\em $T_1$-space} if for any distinct points $x,y\in X$ the point $x$ has a neighborhood $U_x\subset X$ such that $y\notin U_x$;
\item[$T_2$:] {\em Hausdorff\/}   if for any distinct points $x,y\in X$ the point $x$ has a neighborhood $U_x\subset X$ such that $y\notin\overline{U}_x$;
\item[$T_{\frac12 2}$:] {\em semi-Hausdorff\/} if for any distinct points $x,y\in X$ the point $x$ has a neighborhood $U_x\subset X$ such that $y\notin\overline{U}^\circ_x$;
\item[$T_{2\frac12}$:] {\em functionally Hausdorff\/}  if for any distinct points $x,y\in X$ there is a continuous function $f:X\to[0,1]$ such that $f(x)\ne f(y)$;
\item[$sM$:] {\em submetrizable} if $X$ admits a continuous metric;
\smallskip

\item[$R$:] {\em regular} if for any point $x\in X$ and a neighborhood $O_x\subset X$ of $x$ there is a neighborhood $U_x\subset X$ of $x$ such that $\overline{U}_x\subset O_x$;
\item[$\tfrac12 R$:] {\em semiregular} if for any point $x\in X$ and a neighborhood $O_x\subset X$ of $x$ there is a neighborhood $U_x\subset X$ of $x$ such that $\overline{U}^\circ_x\subset O_x$;
\item[$R\tfrac12$:] {\em completely regular} if for any point $x\in X$ and a neighborhood $O_x\subset X$ of $x$ there is continuous function $f:X\to[0,1]$ such that $f(x)=0$ and $f^{-1}\big([0,1)\big)\subset O_x$;
\item[$T_{3\frac12}$:] {\em Tychonoff\/} if $X$ is a completely regular $T_1$-space;
\item[$T_{3}$:] a {\em $T_3$-space} if $X$ is a regular $T_1$-space;
\item[$T_{\frac12 3}$:] a {\em $T_{\frac12 3}$-space} if $X$ is a semi-regular $T_1$-space.
\end{itemize}
\end{definition}
For any topological space these separation axioms relate as follows:
$$\xymatrix{
T_0&T_1\ar@{=>}[l]&T_{\frac12 2}\ar@{=>}[l]&T_{\frac12 3}\ar@{=>}[l]\ar@{=>}[r]&\tfrac12R\\
&&T_2\ar@{=>}[u]&T_3\ar@{=>}[l]\ar@{=>}[r]\ar@{=>}[u]&R\ar@{=>}[u]&\hskip-55pt.\\
&sM\ar@{=>}[r]&T_{2\frac12}\ar@{=>}[u]&T_{3\frac12}\ar@{=>}[l]\ar@{=>}[u]\ar@{=>}[r]&R\tfrac12\ar@{=>}[u]
}
$$
Known (or simple) examples show that none of the implications in this diagram can be reversed.
However for normally quasi-uniformizable spaces the situation changes dramatically.

\begin{theorem}\label{equi} If a topological space $X$ is normally quasi-uniformizable, then
\begin{enumerate}
\item $X$ is completely regular if and only if $X$ is regular if and only if $X$ is semiregular;
\item $X$ is functionally Hausdorff if and only if $X$ is Hausdorff if and only if $X$ is semi-Hausdorff.
\end{enumerate}
\end{theorem}

This theorem follows immediately from Corollary~\ref{cmain}.

Taking into account that the topology of any completely regular space is generated by a uniformity (inherited from the Tychonoff power $\IR^{C(X)}$ where $C(X)$ is the set of all continuous real-valued functions on $X$) and applying Proposition~\ref{p:uT} and Theorem~\ref{equi},
we get the following characterization of completely regular spaces.

\begin{corollary} A semiregular topological space $X$ is completely regular if and only if $X$ is  normally quasi-uniformizable.
\end{corollary}

Now we will apply Theorem~\ref{main} to studying separation axioms in paratopological groups or, more generally, topological monoids with open shifts. Let us recall that a {\em paratopological group} is a group $G$ endowed with a topology making the group multiplication $G\times G\to G$, $(x,y)\mapsto xy$,  continuous.

A {\em topological monoid} is a topological semigroup $S$ possessing a (necessarily unique) two-sided unit $e\in S$. We shall say that a topological monoid $S$ has {\em open shifts} if for any elements $a,b\in S$ the two-sided shift $s_{a,b}:S\to S$, $s_{a,b}:x\mapsto axb$, is an open map.

\begin{remark}
It is clear that each paratopological group is a topological monoid with open shifts. The closed half-line $[0,\infty)$ endowed the Sorgenfrey topology (generated by the base $\mathcal B=\{[a,b):0\le a<b<\infty\}$) and the operation of addition of real numbers is a topological monoid with open shifts, which is not a (paratopological) group.
\end{remark}

Each topological monoid $S$ carries five natural quasi-uniformities:
\begin{itemize}
\item the {\em left quasi-uniformity} $\mathcal L$ generated by the base $\big\{\{(x,y)\in S\times S:y\in xU\}:U\in\mathcal N_e\big\}$;
\item the {\em right quasi-uniformity} $\mathcal R$ generated by the base $\big\{\{(x,y)\in S\times S:y\in Ux\}:U\in\mathcal N_e\big\}$;
\item the {\em two-sided quasi-uniformity} $\mathcal L\vee \mathcal R$ generated by the base $\big\{\{(x,y)\in S\times S:y\in Ux\cap xU\}:U\in\mathcal N_e\big\}$,
\item the {\em Roelcke quasi-uniformity} $\mathcal R\mathcal L$ generated by the base $\big\{\{(x,y)\in S\times S:y\in UxU\}:U\in\mathcal N_e\big\}$, and
\item the {\em quasi-Roelcke uniformity} $\FU=\mathcal R\mathcal L^{-1}\vee\mathcal L\mathcal R^{-1}$ generated by the base\newline $\big\{\{(x,y)\in S\times S:Ux\cap yU\ne\emptyset\ne Uy\cap xU\}:U\in\mathcal N_e\big\}$.
\end{itemize}
Here by $\mathcal N_e$ we denote the family of all open neighborhoods of the unit $e$ in $S$.

The quasi-uniformities $\mathcal L$, $\mathcal R$, $\mathcal L\vee\mathcal R$, and $\mathcal R\mathcal L$ are well-known in the theory of topological and paratopological groups (see \cite[Ch.2]{RD}, \cite[\S1.8]{AT}), whereas the uniformity $\FU$ seems to be news. It should be mentioned that on topological groups the quasi-Roelcke uniformity $\FU$ coincides with the Roelcke quasi-uniformity $\mathcal R\mathcal L$. The following diagram describes the relation between these five quasi-uniformities (an arrow $\U\to\V$ in the diagram indicates that $\U\subset\V$).
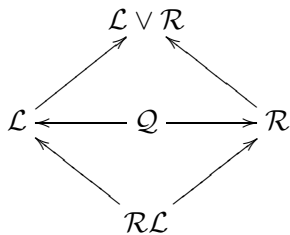
$$\xymatrix{
&\mathcal L\vee\mathcal R\\
\mathcal L\ar[ru]&\FU\ar[r]\ar[l]&\mathcal R\ar[lu]\\
&\mathcal R\mathcal L\ar[lu]\ar[ru]
}$$

If a topological monoid $S$ has open shifts, then the quasi-uniformities $\mathcal L$, $\mathcal R$, $\mathcal L\vee\mathcal R$ and $\mathcal R\mathcal L$ generate the original topology of $S$  (see \cite{Koper}, \cite{KMR}) whereas the quasi-Roelcke uniformity $\FU$ generates a topology $\tau_\FU$, which is (in general, strictly) weaker than the topology $\tau$ of $S$. If $S$ is a paratopological group, then the topology $\tau_\FU$ on $G$ coincides with the join $\tau_2\vee (\tau^{-1})_2$ of the second oscillator topologies considered by the authors in \cite{BR02}. The topology $\tau_\FU$ turns the paratopological group into a quasi-topological group, i.e., a group endowed with a topology in which the inversion and all shifts are continuous (see \cite[7.3]{BR15}).

Observe that in contrast to the classical quasi-uniformities $\mathcal L,\mathcal R,\mathcal R\mathcal L$, and $\mathcal L\vee\mathcal R$, the quasi-Roelcke uniformity $\FU=\FU^{-1}$ is symmetric and so is a uniformity.
By Proposition~\ref{p:uT}, the uniformity $\FU$ is normal (with respect to the topologies $\tau_\FU$ and $\tau$). A similar fact holds for the other three quasi-uniformities.

\begin{proposition}\label{p2} For any topological monoid $S$ with open shifts the quasi-uniformities $\mathcal L$, $\mathcal R$, and $\mathcal R\mathcal L$ are normal.
\end{proposition}

\begin{proof} First we show that the left quasi-uniformity $\mathcal L$ is normal. Given a subset $A\subset S$ and a neighborhood $U\in\mathcal N_e$ we need to show that $\overline{A}\subset \overline{B(A,\mathcal L_U)}^\circ$ where $\mathcal L_U=\{(x,y)\in S\times S:y\in xU\}$. The continuity of (right) shifts in $S$ implies that $$\overline A\cdot U\subset \overline{AU}=\overline{B(A,\mathcal L_U)}.$$
Since the monoid $S$ has open (left) shifts, the set $\overline A\cdot U$ is open in $S$ and hence is contained in the interior $\overline{B(A,\mathcal L_U)}^\circ$ of $\overline{B(A,\mathcal L_U)}$. So, $\overline A\subset \overline A\cdot U\subset\overline{B(A,\mathcal L_U)}^\circ$ and hence the left quasi-uniformity $\mathcal L$ is normal.

By analogy we can prove the normality of the right quasi-uniformity $\mathcal R$ on $S$. The normality of the left quasi-uniformity $\mathcal L$ implies the normality of the Roelcke quasi-uniformity $\mathcal R\mathcal L\subset\mathcal L$.
\end{proof}

\begin{problem} Is the two-sided quasi-uniformity $\mathcal L\vee\mathcal R$ normal for any paratopological group (more generally, any topological monoid with open shifts)?
\end{problem}

Now we shall apply Theorem~\ref{main} and prove the following normality-type property of topological monoids with open shifts.

\begin{theorem}\label{para} Let $S$ be a topological monoid with open shifts. For any neighborhood $U\subset S$ of the unit $e$ of $S$ and any subset $A\subset S$ there exists a continuous function $f:S\to [0,1]$ such that $f(A)\subset \{0\}$ and $f^{-1}\big([0,1)\big)\subset \overline{AU}^\circ\cap \overline{U\kern-1pt A}^\circ$.
\end{theorem}

\begin{proof} Let $\mathcal L_U=\{(x,y)\in S\times S:y\in xU\}\in \mathcal L$ and $\mathcal R_U=\{(x,y)\in S\times S:y\in Ux\}\in \mathcal R$ be the entourages determined by the neighborhood $U$. Applying Theorem~\ref{main} to the normal quasi-uniformities $\mathcal L$ and $\mathcal R$ on $S$, we obtain two continuous functions $f_L,f_R:S\to[0,1]$ such that
$$A\subset f_L^{-1}(0)\subset f^{-1}_L\big([0,1)\big)\subset \overline{B(A,\mathcal L_U)}^\circ=\overline{AU}^\circ\mbox{ \ and \ }A\subset f_R^{-1}(0)\subset f^{-1}_R\big([0,1)\big)\subset \overline{B(A,\mathcal R_U)}^\circ=\overline{UA}^\circ.$$
Then the function $f:=f_L\cdot f_R:S\to[0,1]$ is continuous and has the required property:
$$A\subset f^{-1}(0)\subset f^{-1}\big([0,1)\big)\subset \overline{AU}^\circ\cap \overline{UA}^\circ.$$
\end{proof}

Applying Theorem~\ref{para} to a singleton $A=\{x\}\subset S$, we get the following corollary.

\begin{corollary}\label{mono} Let $S$ be a topological monoid with open shifts. For any point $x\in S$ and neighborhood $O_x\subset S$ of $x$ there exists a continuous function $f:S\to [0,1]$ such that $f(x)=0$ and $f^{-1}\big([0,1)\big)\subset \overline{O}_x^\circ$.
\end{corollary}

It its turn, this corollary implies the following two characterizations.

\begin{corollary}\label{monoreg} A topological monoid $S$ with open shifts is:
\begin{enumerate}
\item completely regular if and only if $S$ is regular if and only if $S$ is semiregular;
\item functionally Hausdorff if and only if $S$ is Hausdorff if and only if $S$ is semi-Hausdorff.
\end{enumerate}
\end{corollary}

Since each paratopological group is a topological monoid with open shift, we can apply Corollary~\ref{monoreg} to paratopological groups and obtain the following two characterizations.
The first of them answers a (more than 50 years) old problem in the theory of paratopological groups (see \cite[Question 1.2]{Rav1}, \cite[Problem 3.1.1]{AT}, \cite[Problem 4.1]{Tka2}, and \cite[Problem 2.1]{Tka}).

\begin{corollary}\label{r1} A paratopological group $G$ is completely regular if and only if $G$ is regular if and only if $G$ is semiregular.
\end{corollary}

We mention that the regularity of semi-regular paratopological groups was first proved by Ravsky \cite{Rav2}.
The other characterization solves an open problem in \cite[Problem 2.4]{Tka}.

\begin{corollary}\label{h1} A paratopological group $G$ is functionally Hausdorff if and only if $G$ is Hausdorff if and only if $G$ is semi-Hausdorff.
\end{corollary}

Corollaries~\ref{r1} and \ref{h1} (and Theorem~\ref{equi}) show that for paratopological groups (more generally, for normally quasi-uniformizable topological spaces) the diagram describing the relation between various separation axioms transforms to the following symmetric form:
$$\xymatrix{
&&T_{\frac12 2}\ar@{<=>}[d]&T_{\frac12 3}\ar@{=>}[l]\ar@{<=>}[d]\ar@{=>}[r]&\tfrac12R\\
T_0&T_1\ar@{=>}[l]&T_2\ar@{=>}[l]&T_3\ar@{=>}[l]\ar@{=>}[r]&R\ar@{<=>}[u]\ar@{<=>}[d]&\hskip-55pt.\\
&&T_{2\frac12}\ar@{<=>}[u]&T_{3\frac12}\ar@{=>}[l]\ar@{<=>}[u]\ar@{=>}[r]&R\tfrac12\\
}
$$
\smallskip

Simple examples presented in \cite[Section 1]{Rav2} or  \cite[Section 2]{Tka} show that for paratopological groups the separation axioms $T_i$, $i\in\{0,1,2,3\}$, are pairwise non-equivalent. On the other hand, for topological groups the separation axiom $T_0$ is equivalent to $T_{3\frac12}$ and consequently to all other separation axioms $T_i$, $i\in\{1,2,\frac122,2\frac12,\frac123,3,3\frac12\}$. The real line endowed with the anti-discrete topology is a regular topological group which fails to satisfy the separation axiom $T_0$.
\smallskip

In fact, the first parts of Corollaries~\ref{h1} and \ref{monoreg}(2) admit an elementary proof exploiting the quasi-Roelcke uniformity. Let $S$ be a Hausdorff topological monoid. We shall say that $S$ is {\em countably Hausdorff\/} if there is a countable family $\U$ of neighborhoods of the unit $e\in G$ such that for any distinct points $x,y\in G$ there is a neighborhood $U\in\U$ such that $Ux\cap yU=\emptyset$ or $xU\cap Uy=\emptyset$. It is clear that each first-countable Hausdorff topological monoid is countably Hausdorff.

The following theorem yields a simple alternative proof of the first parts of Corollaries~\ref{monoreg}(2) and \ref{h1}.

\begin{theorem}\label{Rolk} Any (countably) Hausdorff topological monoid $S$ with open shifts is functionally Hausdorff (and submetrizable).
\end{theorem}

\begin{proof} Since $S$ is Hausdorff, the quasi-Roelcke uniformity $\FU$ is separated and hence induces a Tychonoff topology $\tau_\FU$ on $S$. Since $S$ has open shifts, the identity map $\mathrm{id}:S\to (S,\tau_\FU)$ is continuous, which implies that the space $S$ is functionally Hausdorff.

Now assume that the topological monoid $S$ is countably Hausdorff. For every neighborhood $V\in\mathcal N_e$ of the unit $e$, consider the basic entourage
$$\fe[V]=\{(x,y)\in S\times S:Vx\cap yV\ne\emptyset\ne Vy\cap xV\}$$of the quasi-Roelcke uniformity $\FU$. Since $S$ is countably Hausdorff, we can choose a countable family $\U$ of neighborhoods of the unit $e\in S$ such that  $\bigcap_{V\in\U}\fe[V]=\Delta_X$ and hence the uniform space $(S,\FU )$ admits a uniform (and thus continuous) metric $d:S\times S\to S$ (see Theorems~8.1.10 and 8.1.21 in \cite{Eng}). This implies that the space $S$ is submetrizable.
\end{proof}

Since each first-countable Hausdorff topological monoid is  countably Hausdorff, Theorem~\ref{Rolk} implies the following corollary.

\begin{corollary}\label{last} Any first-countable Hausdorff topological monoid with open shifts is submetrizable.
\end{corollary}

Since each paratopological group is a topological monoid with open shifts, Theorem~\ref{Rolk} implies another corollary, which answers affirmatively Problem 3.11 in \cite{AT} (repeated as Problem 8.7 in \cite{Tka}) and Problem 2.21 in \cite{IS}, and also generalizes many results on the submetrizability of paratopological groups (see, \cite[\S8]{Tka}, \cite{LL}, \cite{XL}, \cite{IS}).

\begin{corollary}\label{last} Any (countably) Hausdorff paratopological group is functionally Hausdorff (and submetrizable).
\end{corollary}

Theorem~\ref{Rolk} and Corollary~\ref{last} motivate the problem of recognizing countably Hausdorff paratopological groups, equivalently, recognizing Hausdorff paratopological groups whose quasi-Roelcke uniformity has countable pseudocharacter $\psi(\FU)=\min\{|\U|:\U\subset\FU,\;\bigcap\U=\Delta_X\}$. This problem is treated in \cite{BR15} where many upper bounds on $\psi(\FU)$ are found.

\begin{remark} In \cite[Example 3.3]{LL} Lin and Liu constructed an example of a Hausdorff  paratopological abelian group of countable pseudocharacter, which is not submetrizable. In \cite[2.15]{IS} S\'anchez constructed a precompact Hausdorff paratopological abelian group with countable pseudo-character, which admits no injective continuous map into a first-countable space. Finally in \cite[8.1]{BR15} the authors constructed an example of a zero-dimensional (and hence regular) Hausdorff abelian group $G$ of countable pseudocharacter which is not submetrizable (and has no $G_\delta$-diagonal under MA$+\neg$CH). These examples show that Corollary~\ref{last} cannot be generalized to Hausdorff paratopological groups of countable pseudocharacter.
\end{remark}

\section*{Acknowledgements}

This paper was initiated during the visit of the authors to the University of Hradec Kr\'alov\'e (Czech Republic) in September 2014. The authors express their sincere thanks to this university (and personally to Du\v san Bedna\v r\'\i k) for the hospitality. The authors thank  Li-Hong Xie and Peng-Fei Yan for the suggestion to generalize our initial results beyond the class of paratopological groups (to the class of topological monoids with open shifts). The second author also thanks his Lady for the inspiration.
%\newpage

\end{document}